\documentclass[11pt]{article}

\usepackage{girth}
\usepackage{abbreviations}
\usepackage{amsmath}
\usepackage{amssymb}
\usepackage{amstext}
\usepackage{amscd}
\usepackage{amsthm}
\usepackage{makeidx}
\usepackage{graphicx}
\usepackage{hyperref}
\usepackage{supertabular}
\usepackage{enumerate}
\usepackage{titling}
\usepackage{wasysym}
\usepackage{comment}
\usepackage{cleveref}
\usepackage{listings}
\usepackage{csquotes}

\usepackage{microtype}

\begin{document}
\title{High-girth near-Ramanujan graphs with lossy vertex expansion}
\author{Theo McKenzie\thanks{\texttt{mckenzie@math.berkeley.edu}. This material is based upon work supported by the National Science Foundation Graduate Research Fellowship Program under Grant No. DGE-1752814. Any opinions, findings, and conclusions or recommendations expressed in this material are those of the authors and do not necessarily reflect the views of the National Science Foundation.} \\ University of California, Berkeley\and Sidhanth Mohanty\thanks{\texttt{sidhanthm@cs.berkeley.edu}. Supported by NSF grant CCF-1718695.}\\ University of California, Berkeley}

\maketitle

\begin{abstract}
    Kahale \cite{Kah} proved that linear sized sets in $d$-regular Ramanujan graphs have vertex expansion at least $\frac{d}{2}$ and complemented this with construction of near-Ramanujan graphs with vertex expansion no better than $\frac{d}{2}$.  However, the construction of Kahale encounters highly local obstructions to better vertex expansion.  In particular, the poorly expanding sets are associated with short cycles in the graph.  Thus, it is natural to ask whether the vertex expansion of \emph{high-girth} Ramanujan graphs breaks past the $\frac{d}{2}$ bound.  Our results are two-fold:
    \begin{enumerate}
        \item For every $d = p+1$ for prime $p$ and infinitely many $n$, we exhibit an $n$-vertex $d$-regular graph with girth $\Omega(\log_{d-1} n)$ and vertex expansion of sublinear sized sets bounded by $\frac{d+1}{2}$ whose nontrivial eigenvalues are bounded in magnitude by $2\sqrt{d-1}+O\left(\frac{1}{\log n}\right)$.
        \item In any Ramanujan graph with girth $C\log n$, all sets of size bounded by $n^{0.99C/4}$ have near-lossless vertex expansion $(1-o_d(1))d$.
    \end{enumerate}
    The tools in analyzing our construction include the nonbacktracking operator of an infinite graph, the Ihara--Bass formula, a trace moment method inspired by Bordenave's proof of Friedman's theorem \cite{Bor}, and a method of Kahale \cite{Kah} to study dispersion of eigenvalues of perturbed graphs.
\end{abstract}

\setcounter{page}{0}
\thispagestyle{empty}
\newpage

\section{Introduction}
This paper is concerned with expander graphs, which are ubiquitous in theoretical computer science.  A natural and highly well-studied quantity associated with a $d$-regular graph is its \emph{edge expansion} defined as
\[
    \min_{|S|\leq \epsilon n} E(S,\overline S)/|S|,
\]
namely the minimum ratio of edges leaving a set $S$ to the size of $S$ for all $S$ of appropriately bounded size.  While edge expansion is known to be intractable to compute, there are explicit constructions of good edge expanders, and it is closely related to the second largest magnitude eigenvalue of its adjacency matrix, also known as \emph{spectral expansion} of a graph, via the expander mixing lemma and Cheeger's inequality \cite{Che1}. Spectral expansion is easily computable.  In particular, an application of the expander mixing lemma proves that small enough sets in graphs with spectral expansion $o(d)$ have near-optimal edge expansion of $(1-o_d(1))d$.

A natural analog to edge expansion is \emph{vertex expansion}, defined as 
\[
    \min_{|S|\leq \epsilon n} |\Gamma(S)|/|S|
\]
for some constant $\epsilon$, where $\Gamma(S)$ is the neighborhood of the set $S$ (potentially containing vertices of $S$). However, as difficult as edge expansion is to ascertain, vertex expansion has proven far more challenging.

As witnessed by balls around a vertex, we cannot hope for vertex expansion greater than $d-1$. Therefore we call a graph a \textit{lossless vertex expander} if for every $\delta$, there exists an $\epsilon$ such that there is vertex expansion $d-1-\delta$ for sets of size $\epsilon n$.  Lossless vertex expanders exist since a random $d$-regular graph is one with high probability (see \cite[Theorem 4.16]{HLW} for a proof).  However no deterministic construction of such graphs is known.  In an effort to understand lossless vertex expansion better and give explicit constructions, a natural question to ask is:
\begin{displayquote}
    \emph{What properties of random graphs leads to lossless vertex expansion?}
\end{displayquote}
Since a random $d$-regular graph is near-Ramanujan with high probability \cite{Fri}, and since near-Ramanujan graphs have near-optimal edge expansion, it is natural to inquire if spectral expansion has any implications for vertex expansion as well.  Kahale \cite{Kah} showed that the spectral expansion gives a bound on the vertex expansion. Specifically, Ramanujan graphs (namely graphs with optimal spectral expansion) have vertex expansion at least $d/2$.  While this is a nontrivial implication, it falls short of achieving the coveted \emph{losslessness} property.  Kahale also proved that the bound of $d/2$ is tight.  In particular, he exhibited an infinite family of near-Ramanujan graphs with vertex expansion $d/2$, which means spectral expansion alone is not sufficient for lossless vertex expansion.

The occurrence of a copy of $K_{2,d}$\footnote{complete bipartite graph with $2$ vertices on one side and $d$ vertices on the other} as a subgraph is the obstruction to lossless vertex expansion in Kahale's example.  Kahale's example deviates from a random graph in that it is highly unlikely for a random graph to contain a copy of $K_{2,d}$ as a subgraph.  More generally, random graphs have the property that with high probability any two ``short'' cycles are far apart, which Kahale's example doesn't satisfy.  Thus, it is natural to ask if the ``near-Ramanujan'' property in conjunction with the ``separatedness of cycles'' property of random graphs break past the $d/2$ barrier of Kahale.  The ``separatedness of cycles'' property is especially interesting to consider since it is a key property of random graphs exploited in proofs of Alon's conjecture \cite{Fri,Bor}.  A concrete question we can ask is:
\begin{displayquote}
    \emph{Do Ramanujan graphs with $\Omega(\log_{d-1} n)$ girth have lossless vertex expansion?}
\end{displayquote}
An affirmative answer to the above question would prove that the Ramanujan graphs of Lubotzky, Phillips, and Sarnak \cite{LPS} are lossless vertex expanders.  Towards answering the above question, we prove the following negative result:
\begin{theorem} \label{thm:main-negative}
    For every $d = p+1$ for prime $p$, there is an infinite family of $d$-regular graphs $G$ on $n$ vertices of girth $\ge\left(\frac{2}{3}-o_n(1)\right)\log_{d-1} n$ where there is a set of vertices $U$ such that $|\Gamma(U)|\le(d+1)|U|/2$, $|U|\le n^{1/3}$, and $\max\{\lambda_2(G),-\lambda_n(G)\} \le 2\sqrt{d-1}+O(1/{\log_{d-1} n})$.
\end{theorem}
We also complement the above with a positive result which can be summarized as ``small enough sets in Ramanujan graphs expand nearly losslessly'':
\begin{theorem} \label{thm:main-positive}
    Let $G$ be a $d$-regular Ramanujan graph with girth $C\log_{d-1}n$, then every set of $S$ of size $\le n^{\kappa}$ for $\kappa < \frac{C}{4}$ has vertex expansion $(1-o_d(1))d$.
\end{theorem}

\subsection{Technical overview}
We give a brief description of how \pref{thm:main-negative} and \pref{thm:main-positive} are proved.
\paragraph{Overview of proof of \pref{thm:main-negative}.}  Our proof is inspired by that of Kahale's.  At a high level, Kahale embeds a copy of $K_{2,d}$ within a Ramanujan graph.  We proceed similarly to Kahale, but instead of embedding a $K_{2,d}$, we embed a single subgraph $H$ that is high girth but a lossy vertex expander and show that if $H$ has size $n^\alpha$ for some $0 < \alpha\leq 1/3$, the overall graph is still near-Ramanujan.

Our proof involves two steps: the first step is in proving that the subgraph $H$ being embedded has spectral radius bounded by $2\sqrt{d-1}$, and the second step is in proving that planting $H$ within a Ramanujan graph results in a near-Ramanujan graph.  For the first step, we describe an infinite graph containing $H$ and bound its spectral radius via a trace moment method.  The trace moment method involves bounding the number of closed walks satisfying certain properties within a graph, and is inspired by an encoding argument from Bordenave's proof of Friedman's theorem \cite{Bor}.

The second step is in proving that our method of embedding a copy of $H$ within a Ramanujan graph does not perturb the eigenvalues by a large amount. Towards doing so, we use the fact that the spectral radius of $H$ is bounded by $2\sqrt{d-1}$ in conjunction with Kahale's argument about dispersion of eigenvalues in high-girth graphs.

\paragraph{Overview of proof of \pref{thm:main-positive}.}  We first prove that if a set $S$ in a Ramanujan graph has ``lossy'' vertex expansion, then we can construct a graph $H$ on vertex set $S$ such that (i) the girth of $H$ is at least half the girth of $G$, and (ii) the average degree of $H$ is ``high'' (in particular, the worse the vertex expansion of $S$, the higher the average degree of $H$).  We then employ the irregular Moore bound, which gives a quantitative tradeoff between the average degree of a graph and its girth.  In particular, this would imply that a Ramanujan graph with ``lossy'' vertex expansion necessarily must have ``low'' girth.


\subsection{Related work}
\paragraph{Applications of vertex expanders.}
There are many applications of expander graphs where having vertex expansion is particularly useful. For example, lossless expanders are particularly of interest in the field of error correcting codes \cite{LMSS,SipSpiel,Spiel}. Lossless vertex expanders give linear error correcting codes that are decodable in linear time \cite{SipSpiel}. Guruswami, Lee and Razborov \cite{GLR} use bipartite vertex expanders to construct large subspaces of $\R^n$ where all vectors $x$ in the subspace satisfy $(\log n)^{-O(\log \log \log n)}||x||_2\leq ||x||_1\leq \sqrt n||x_2||$.

\paragraph{Explicit constructions.}  Constructions of Ramanujan graphs of \cite{LPS,Mar88,Mor} of all degrees that are of the form $p^r+1$ for $p$ prime, as well as the construction of near-Ramanujan graphs of every degree of \cite{MOP} have vertex expansion $\sim\frac{d}{2}$ just by virtue of being Ramanujan via Kahale's result.  In fact no deterministic construction has improved upon the $d/2$ bound obtained from solely spectral information.    In a remarkable work, Capalbo et. al. \cite{CRVW} exhibited an explicit construction of a bipartite graph where subsets of one side of the bipartite graph expand losslessly to the other, using a zig-zag product so the the losslessness of a small, random-like graph boosts the expansion from a large, potentially lossy vertex expanding graph.

\paragraph{Quantum Ergodicity.}  Quantum ergodicity is another area where both local and global properties of random-like graphs are used.  In particular, Anantharaman and Le Masson \cite{AL} proved that graphs that have few short cycles (and are therefore close to high girth) and spectral expansion are quantum ergodic, which in this context means the eigenvectors are equidistributed across vertices. Anantharaman, as well as Brooks, Le Masson, and Lindenstrauss exhibited alternative proofs \cite{Anth, BLL}. The proof from \cite{BLL} shows that quantum ergodicity is equivalent to the mixing of a certain graphical operator. They then use high girth to show that this is equivalent to showing mixing on the infinite tree, then expansion to show the nonbacktracking operator mixes on the tree.

\paragraph{Eigenvector delocalization.} Ganguly and Srivastava, and later Alon, Ganguly and Srivastava \cite{GS,AGS} give a perturbation of the LPS graph similar to Kahale's argument, but instead of individual vertices, two trees are added and connected to the graph. By assuming the tree is sufficiently deep and carefully connecting the tree to the rest of the graph, the authors create a graph that is high girth but contains eigenvectors that are localized. These graphs are also lossy vertex expanders. However, they show that these graphs cannot be Ramanujan, but rather have spectral radius at least $(2+c)\sqrt{d-1}$ where $c>0$ is a constant. Alon \cite{Alon} used eigenvector delocalization to create near-Ramanujan expanders of every degree by perturbing known constructions of Ramanujan or near-Ramanujan graphs.  Paredes \cite{PP} used similar techniques to remove short cycles in a graph while preserving expansion and uses this to algorithmically create graphs that are near-Ramanujan and also have girth at least $\Omega(\sqrt{\log n})$.

\paragraph{Complexity of constraint satisfaction problems.}  Proofs that it is hard for even linear degree Sum-of-Squares to refute random 3XOR and 3SAT instances on $n$ variables \cite{Gri,Sch} rely on lossless vertex expansion of some sets in a graph underlying a random instance, which suggests a connection between deterministic algorithms for constructing lossless vertex expanders and algorithms for explicit hard instances for Sum-of-Squares.

\section{Preliminaries}
\subsection{Elementary graph theory}
\begin{definition}
The \textit{girth} $g(G)$ of a graph $G$ is the length of the smallest cycle in $G$.
\end{definition}

\begin{definition}
For $G=(V,E)$, the \textit{valency} of $a\in V$ to $B\subset V$ is $|\Gamma(a)\cap B|$, where $\Gamma(S)$ for $S\subset V$ is the set of neighbors of $S$ in $G$.
\end{definition}

\begin{definition}
    The \textit{ball} of radius $h$ around a set $U\subset V$, denoted $\textnormal{Ball}_h(U)$, is the set of vertices of distance at most $h$ from $U$.
\end{definition}

\begin{definition}
The \textit{vertex expansion} of a set $U\subset V$ is
\[
\Psi(U):=\frac{|\Gamma(U)|}{|U|}.
\]
Similarly, the $\epsilon$-vertex expansion of a graph $G$ is:
\[
\Psi_{\epsilon}(G)=\min_{|U|\leq \epsilon|V|}\Psi(U)
\]
where $U$ ranges over subsets of $V$, and $\epsilon$ is an arbitrary constant.
\end{definition}

\begin{definition}
    Given a graph $G$, we use $A_G$ to denote its adjacency matrix.  When $G$ is a finite graph on $n$ vertices, the eigenvalues of $A_G$ can be ordered as $\lambda_1(G)\ge\lambda_2(G)\ge\dots\ge\lambda_n(G)$.
\end{definition}

\begin{definition}
    We use $B_G$ to denote the \emph{nonbacktracking matrix} of a graph $G$ which is a matrix with rows and columns indexed by directed edges of $G$ defined as follows:
    \[
        B[(u,v),(w,x)] =
        \begin{cases}
            1   &\text{if $v = w$ and $u \ne x$}\\
            0   &\text{otherwise.}
        \end{cases}
    \]
\end{definition}

\begin{definition}
    The \emph{spectral expansion} of a finite graph $G$, denoted $\lambda(G)$ is defined as $\max\{\lambda_2(G), -\lambda_n(G)\}$, which can equivalently be described as the ``second largest absolute eigenvalue''.
\end{definition}
We now state the following standard fact known as the \emph{expander mixing lemma} (see \cite[Lemma 2.5]{HLW}).
\begin{lemma}[Expander Mixing Lemma]    \label{lem:expander-mixing-lemma}
    Let $G$ be a $d$-regular graph on $n$ vertices.  For any two subsets of vertices, $S,T\subseteq V(G)$, let $e(S,T)$ be the number of pairs of vertices $(x,y)$ such that $x\in S, y\in T$ and $\{x,y\}$ is an edge in $G$.  Then:
    \[
        \left|e(S,T) - \frac{d}{n}|S|\cdot|T|\right| \le \lambda(G)\sqrt{|S|\cdot|T|}.
    \]
\end{lemma}
And finally, we state the ``irregular Moore bound'' of \cite{AHL} which articulates a tradeoff between the average degree of a graph and its girth.
\begin{lemma}   \label{lem:irreg-moore-bound}
    Let $G$ be a $n$-vertex graph with average degree-$d$.
    Then
    \[
        g(G) \le 2\log_{d-1} n + 2.
    \]
\end{lemma}

\subsection{Operator theory}
In this section, let $V$ be a countable set and $T:\ell_2(V)\to\ell_2(V)$ be a bounded linear operator.
\begin{definition}
    The \emph{spectrum} of $T$, which we denote $\mathsf{spec}(T)$, is the set of all $\lambda\in\C$ such that $\lambda\Id-T$ is not invertible.
\end{definition}

\begin{definition}
    The \emph{spectral radius} of $T$, which we denote $\rho(T)$ is defined as $\sup\{|\lambda|:\lambda\in\mathsf{spec}(T)\}$.
\end{definition}

\begin{fact}
    The \emph{operator norm} of $T$, which we write as $\|T\|$ is equal to $\sqrt{\rho(TT^*)}$ where $T^*$ is the adjoint of $T$.\footnote{Since $\ell_2(V)$ comes equipped with the inner product $\langle f, g\rangle \coloneqq \sum_{v\in V}f(v)g(v)$, $T^*$ is simple the ``transpose'' of $T$.}
\end{fact}

\begin{fact}
    $\rho(T) = \lim_{\ell\to\infty} \|T^\ell\|^{1/\ell}$.
\end{fact}

\begin{fact}[{Consequence of \cite[Theorem 6]{Que96}}]  \label{fact:specrad-max-basis}
    Suppose $T$ is a self-adjoint operator, and $\Phi$ is a basis of $\ell_2(V)$.  Then:
    \[
        \rho(T) = \sup_{\phi\in\Phi}\limsup_{k\to\infty}|\langle \phi, T^k\phi\rangle|^{1/k}.
    \]
\end{fact}

\begin{fact}
    Let $A$ be any principal submatrix of $T$.  Then $\rho(A)\le\rho(T)$.
\end{fact}

\begin{corollary}   \label{cor:subgraph-less-specrad}
    If $H$ is a subgraph of (possibly infinite) graph $G$, then $\rho(A_H)\le\rho(A_G)$.
\end{corollary}

\section{Infinite trees hanging from a biregular graph}
Let $H$ be any $(2,d-1)$-biregular graph where the partition with degree-$(d-1)$ vertices is called $U$ and the partition with degree-$2$ vertices is called $V$.  Let $X$ be the infinite graph constructed from $H$ in the following way:
\begin{displayquote}
    At every vertex in $U$, the $(d-1)$-regular partition, glue an infinite tree where the root has degree-$1$ and the remaining vertices have degree-$d$.  At every vertex in $V$, the $2$-regular partition, glue an infinite tree where the root has degree-$(d-2)$ and every other vertex has degree-$d$.
\end{displayquote}
Note that $X$ is a $d$-regular infinite graph.  The main result of this section is:
\begin{lemma}   \label{lem:inf-graph-specrad}
    $\rho(A_X)\le 2\sqrt{d-1}$.
\end{lemma}
To prove \pref{lem:inf-graph-specrad}, we instead turn our attention to the nonbacktracking matrix of $X$, called $B_X$.  In particular, we bound $\rho(B_X)$ and then employ the Ihara--Bass formula of \cite{AFH} for infinite graphs to translate the bound on $\rho(B_X)$ into a bound on $\rho(A_X)$.

Thus, we first prove:
\begin{lemma}   \label{lem:nb-matrix-bound}
    $\rho(B_X)\le\sqrt{d-1}$.
\end{lemma}
We use the following version of the Ihara--Bass formula of \cite{AFH} for infinite graphs.
\begin{theorem} \label{thm:inf-IB}
    Let $G$ be a (possibly infinite) graph.  Then
    \[
        \mathsf{spec}(B_G) = \{\pm 1\}\cup\{\lambda:(D_G-\Id)-\lambda A_G + \lambda^2\Id\text{ is not invertible}\}.
    \]
\end{theorem}
An immediate corollary that we will use is:
\begin{corollary}   \label{cor:nb-bound-to-adj}
    Let $G$ be a $d$-regular graph.  Then $\rho(B_G)\le\sqrt{d-1}$ implies that $\rho(A_G)\le 2\sqrt{d-1}$.
\end{corollary}
\begin{proof}
    If there is $\mu$ in $\mathsf{spec}(A_G)$ such that $|\mu|>2\sqrt{d-1}$, then $\mu\Id - A_G$ is not invertible.  Consequently, by \pref{thm:inf-IB} $\lambda = \frac{\mu+\sqrt{\mu^2-4(d-1)}}{2}$, which is greater than $\sqrt{d-1}$, is in $\mathsf{spec}(B_G)$.
\end{proof}

In light of \pref{cor:nb-bound-to-adj}, we see that \pref{lem:nb-matrix-bound} implies \pref{lem:inf-graph-specrad}.

Towards proving \pref{lem:nb-matrix-bound}, we first make a definition.
\begin{definition}
    We call a walk $W$ a $(a\times b)$-linkage if it can be split into $a$ segments, each of which is a length-$b$ nonbacktracking walk.
\end{definition}

\begin{proof}[Proof of \pref{lem:nb-matrix-bound}.]
    The first ingredient in the proof is the fact that for any $\ell \ge 0$,
    \[
        \rho(B_X)^\ell \le \|B_X^\ell\|
    \]
    and thus
    \[
        \rho(B_X) \le \limsup_{\ell\to\infty} \|B_X^{\ell}\|^{1/\ell}
    \]
    Since $\|B_X^\ell\| = \sqrt{\|B_X^\ell(B_X^*)^\ell\|} = \sqrt{\rho(B_X^\ell(B_X^*)^\ell)}$ it suffices to bound $\rho(T)$ where $T := B_X^\ell(B_X^*)^\ell$ is a bounded self-adjoint operator, and hence by \pref{fact:specrad-max-basis}:
    \[
        \rho(T) = \max_{uv\in \vec{E}(X)} \limsup_{k\to\infty}|\langle 1_{uv}, T^{k}1_{uv}\rangle|^{1/k}.
    \]
    The quantity $\langle 1_{uv}, T^k 1_{uv}\rangle$ is bounded by the number of $(2k\times(\ell+1))$-linkages that start and end at vertex $u$, which we can bound via an encoding argument.  In particular, we will give an algorithm to uniquely encode such linkages and bound the total number of possible encodings.

    \paragraph{Encoding linkages.} Each length-$(\ell+1)$ nonbacktracking segment can be broken into $3$ consecutive phases (of which some can possibly be empty): the phase where distance to $H$ decreases on each step (Phase 1), the second phase where distance to $H$ does not change on each step (Phase 2), and the third phase where distance to $H$ increases on each step (Phase 3).  We further break the third phase into two (possibly empty) subphases --- the first subphase where the distance to $u$ decreases on each step (Phase 3a), and the second subphase where the distance to $u$ increases on each step (Phase 3b).

    To encode the linkage, for each length-$(\ell+1)$ nonbacktracking we specify four numbers denoting the lengths of Phases $1$, $2$, $3$a, and $3$b.  Note that Phase $2$ is nonempty only if it is contained in $H$.  For each step $ab$ in Phase $2$ that goes from $U$ (the $(d-1)$-regular partition) to $V$ (the $2$-regular partition) we specify a number $i$ in $[d-1]$ such that $b$ is the $i$th neighbor of $a$ within $H$.  If the first step $ab$ in Phase $2$ is from $V$ to $U$ we specify a number in $[2]$ denoting if $b$ is the first or second neighbor of $a$.  For each step $ab$ in Phase $3$b we specify a number $i$ in $[d-1]$ such that $b$ is the $i$th neighbor of $a$ that does not lie in the path between between $u$ and $H$.

    \paragraph{Recovering linkages from encodings.} We recover a linkage from its encoding ``segment-by-segment''.  Suppose the first $t$ segments have been recovered, we show how to recover the $(t+1)$-th segment.  Let $x$ be the vertex the walk is at after it has traversed the first $t$ segments.  The steps taken in Phase $1$ can be recovered from the length of the Phase since there is a unique path from any vertex to $H$.  The steps in Phase $2$ alternate between stepping from $V$ to $U$ and from $U$ to $V$.  It is easy to recover the first step of Phase $2$ as well as any step from $U$ to $V$; a step $ab$ from $V$ to $U$ that is not the first step of Phase $2$ is uniquely determined by the previous step, since $a$ has $2$ neighbors in $U$ and by the nonbacktracking nature of the walk there is only one choice for $b$.  Note that Phase 3a is nonempty only if $u$ is not in $H$ and all the steps are contained in the same branch as $u$.  Since there is a unique shortest path between the start vertex of Phase 3a and $u$, the steps taken in Phase 3a can be recovered from its length.  Finally, it is easy to recover the steps taken in Phase 3b since they are explicitly given in the encoding.

    \paragraph{Counting encodings.} Now we turn our attention to bounding the total number of encodings.  For given $\alpha,\beta \ge 0$ such that $\alpha+\beta = 2k(\ell+1)$ we first bound the number of walks such that $\alpha$ steps occur in Phase $2$ (i.e. are within $H$) and $\beta$ steps occur outside Phase $2$ (i.e. are outside $H$).  Let $v_1,v_2,\dots, v_{2k(\ell+1)}$ be the sequence of vertices visited by the walk in order.  Since $d(v_1,H) = d(v_{2k(\ell+1)},H)$, $|d(v_i,H)-d(v_{i+1},H)|\le 1$ always and $|d(v_i,H)-d(v_{i+1},H)|=0$ for every step in Phase $2$, the number of steps of the walk that occur in Phase $3$ of their respective segments is at most $\frac{\beta}{2}$.  In particular, the number of steps that occur in Phase $3$b of their respective segments is bounded by $\frac{\beta}{2}$.  The following bounds hold:
    \begin{itemize}
        \item The number of possible encodings of the lengths of phases is bounded by $(\ell+1)^{8k}$.
        \item The number of possible encodings of the first step of Phase $2$ of each segment is bounded by $2^{2k}$.
        \item The number of possible encodings of the list of $U$-to-$V$ steps in Phase $2$ is bounded by $(d-1)^{\frac{\alpha+1}{2}}$ because the steps taken in Phase 2 alternate between going from $V$ to $U$ and from $U$ to $V$.
        \item The number of possible encodings of the list of steps in Phase $3$b is bounded by $k(\ell+1)(d-1)^{\frac{\beta}{2}}$.
    \end{itemize}
    The above bounds combined with the fact that there are at most $2k\ell$ choices for $(\alpha,\beta)$ pairs gives a bound on the number of $(2k\times(\ell+1))$-linkages of
    \[
        2(k(\ell+1))^2(\ell+1)^{8k}2^{2k}(d-1)^{\frac{\alpha+1}{2}}(d-1)^{\frac{\beta}{2}} \le 2(k(\ell+1))^2(\ell+1)^{8k}2^{2k}\sqrt{d-1}^{2k(\ell+1)+1}.
    \]
    Thus,
    \[
        \rho(T) \le \limsup_{k\to\infty} \left(2(k(\ell+1))^2(\ell+1)^{8k}2^{2k}\sqrt{d-1}^{2k(\ell+1)+1}\right)^{1/k} = 4(\ell+1)^8\sqrt{d-1}^{2(\ell+1)}
    \]
    Consequently,
    \[
        \rho(B_X) \le \limsup_{\ell\to\infty} \rho(T)^{1/2\ell} \le \limsup_{\ell\to\infty} \left(4(\ell+1)^8\sqrt{d-1}^{2(\ell+1)}\right)^{1/{2\ell}} = \sqrt{d-1}.
    \]
\end{proof}

\section{High-girth near-Ramanujan graphs with lossy vertex expansion}
We will plant a high girth graph with low spectral radius within a $d$-regular Ramanujan graph. We will show that such a construction is a spectral expander, but has low vertex expansion. By $u\sim_G v$, we mean that $u$ and $v$ are adjacent in the graph $G$. We will write $u\sim v$ when the graph is clear from context.

Consider a $(2,d-1)$ biregular bipartite graph $H=(U,V,E)$, with vertex components $U$ and $V$. $U$ is the degree-$(d-1)$ component and $V$ the degree-$2$ component. Therefore if we define $\gamma:=|U|$, requiring $\gamma$ to be even, then $|V|=(d-1)\gamma/2$. Call the vertices of $U$ and $V$ $\{u_1,\ldots,u_\gamma\}$ and $\{v_1,\ldots, v_{\gamma(d-1)/2}\}$, respectively. We connect $U$ and $V$ in such a way to maximize the girth of $H$.

\begin{lemma}\label{lem:hgirth}
\[
g(H)\geq 2\log_{d-1}\gamma.
\]
\end{lemma}
\begin{proof}
Because of the valency conditions on $H$, there is a graph $\wt H$ on $\gamma$ vertices $\{\wt u_1,\ldots,\wt u_\gamma\}$, where $\wt u_i\sim_{\wt H} \wt u_j$ if and only if $\exists v_k\in H$ such that $u_i\sim_{H} v_k$ and $u_j\sim_H v_k$. Namely, $U$ corresponds to the vertex set of $\wt H$, and $V$ corresponds to the edge set. $\wt H$ is $d-1$ regular, and, as paths in $\wt H$ of length $r$ correspond to paths of length $2r$ in $H$, $g(H)=2g(\wt H)$.  

By a result of Linial and Simkin \cite{LS}, there exists a graph $\wt H$ that has girth at least $c\log_{d-2}\gamma$, for any $c\in (0,1)$, assuming $\gamma$ is even. Therefore by setting $c=\log(d-2)/\log(d-1)$, we have that $g(\wt H)\geq \log_{d-1}\gamma$ and $g(H)\geq 2\log_{d-1}\gamma$.
\end{proof}

We add a new set of vertices $Q=\{q_1,\ldots, q_\gamma\}$ and add a matching between $Q$ and $U$, adding the edge $q_iu_i$ for $1\leq i\leq \gamma$. Similarly, we add another set of vertices $R=\{r_{i,j}\}, 1\leq i \leq \gamma(d-1)/2,1\leq j\leq d-2$. For each $1\leq i\leq \gamma(d-1)/2$, we then add an edge from $v_i$ to each of $r_{i,j}$ for $1\leq j\leq (d-2)$.

We call $H'$ the graph on $U\cup V \cup Q\cup R$. At this point vertices of $U$ and $V$ have degree-$d$, and vertices of $Q$ and $R$ have degree-$1$. Also, note $\Psi(U)=(d+1)/2$. 
We wish to embed $H'$ into a larger, high girth expander, and show that this new graph maintains high girth and expansion, even though the set $U$ is a lossy vertex expander.  Our argument follows that of \cite[Section 5]{Kah}, but instead of embedding individual vertices, we will embed $H'$. 

\begin{figure}
  \[\includegraphics[height=7cm]{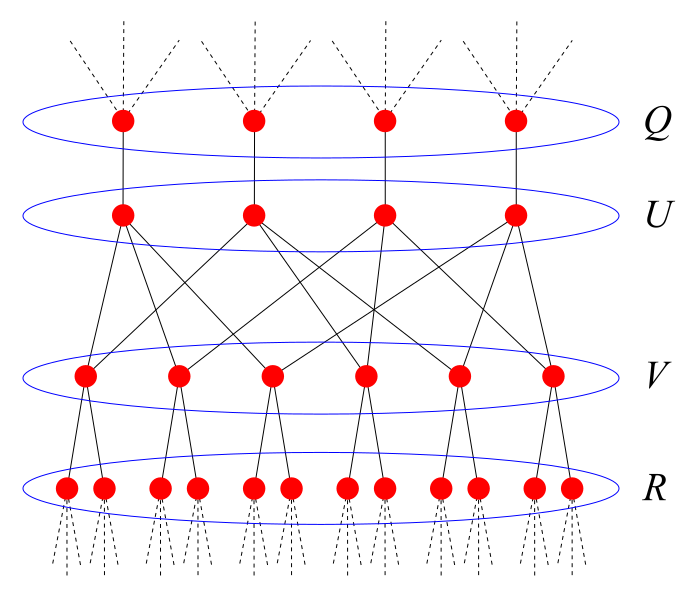}\]
  \caption{$H'$, with labeled components for $d=4$, $\gamma=4$. Note that $\Psi(U)=(d+1)/2$. To create $G'$, we connect $Q$ and $R$ to a well spaced matching in $G$.}
  \label{fig:hprime}
\end{figure}

\begin{theorem}
    For every $d = p+1$ for prime $p\geq 3$, there is an infinite family of $d$-regular graphs $G_m=(V_m,E_m)$ on $m$ vertices, such that $\exists U_m\subset V_m$  with $\Psi(U_m)=(d+1)/2$ for $|U_m|\leq m^{1/3}$, $g(G_m)=(\frac23-o_m(1))\log_{d-1} m$, and such that $\lambda(G_m)\leq 2\sqrt{d-1}+O(1/{\log_{d-1} m})$.
\end{theorem}
\begin{proof}
By the result of Lubotzky, Phillips and Sarnak \cite{LPS}, for such $d$, there exists an infinite family of $d$-regular graphs, where graphs of $n$ vertices have girth $(\frac43-o_n(1)) \log_{d-1}n$ and have spectral expansion  $\leq 2\sqrt{d-1}$.

For a given graph $G=(V,E)$ of this type of size $n$, we attach $H'$ by removing a matching $M\subset E$, $M=\{(a_{1,1},a_{1,2}),\ldots, (a_{k,1},a_{k,2})\}$ for
\begin{equation}\label{eq:kbound}
    k:=\gamma(d-1)(2+(d-1)(d-2))/4.
\end{equation}

We take a matching such that the pairwise distance between edges in the matching is maximized in $G$.

\begin{lemma}\label{lem:spacing}
In a $d$-regular graph on $n$ vertices, there exists a matching $M$ of size $k$ such that for every pair of edges $(a_{i_1,1},a_{i_1,2}),(a_{i_2,1},a_{i_2,2})\in M$, $i_1\neq i_2$,
\[
d((a_{i_1,1},a_{i_1,2}),(a_{i_2,1},a_{i_2,2}))\geq \log_{d-1} n-\log_{d-1}\gamma-O_n(1).
\]
\end{lemma}
\begin{proof}
For a given pair of adjacent vertices $(a_{i,1},a_{i,2})$, as our graph is $d$ regular, there are at most $1+d\frac{(d-1)^r-1}{d-2}$ vertices at distance at most $r$ from $a_{i,1}$, and at most $(d-1)^r$ vertices at distance $r$ from $a_{i,2}$ and distance $r+1$ from $a_{i,1}$. Therefore for any $d\geq 4$, the number of edges at distance at most $r$ from a given edge is less than $4(d-1)^r$. We then greedily add edges by choosing an arbitrary edge with vertices at distance at least $r$ away from all already chosen edges. A $k$th such edge will exist as long as $4k(d-1)^r\leq n$. For our $k$ given in (\ref{eq:kbound}) we can set $r=\log_{d-1} n-\log_{d-1}\gamma-O_n(1)$.  
\end{proof}

To connect $H'$ to $G$, we first delete the matching $M$. Then for every vertex of $Q$ and $R$, we add $d-1$ edges to the set of vertices of $M$, connecting to each vertex of $M$ exactly once. Namely, the induced subgraph on $(Q\cup R)\cup M$ is a $(d-1,1)$ biregular bipartite graph. Call $G'=(V',E')$ the new graph formed from $G$ and $H'$.

We wish to show that $G'$ remains high girth and a good spectral expander. For the girth of $G'$, cycles are either completely contained in $H'$, completely contained in $G$, or a mix between the two. Cycles in $H'$ have length at least $2\log_{d-1}\gamma$ by \pref{lem:hgirth}. Cycles in $G$ have length at least $(\frac43-o_n(1))\log_{d-1}n$ by the construction of \cite{LPS}. For cycles that are a mix of $H'$ and $G$, we must go from one vertex of $H'$ to another vertex of $H'$ through $G$. Therefore by \pref{lem:spacing}, the length of such a cycle is at least $\log_{d-1}n-\log_{d-1}\gamma-O_n(1)$, giving \[g(G')\geq \min\{2\log_{d-1}\gamma,\log_{d-1}n-\log_{d-1}\gamma-O_n(1)\}.\]

To show that the spectrum is not adversely affected, we follow the argument of \cite[Theorem 5.2]{Kah}, with some adjustments. For our new graph, assume that there is an eigenvector $g\perp\bf1$ corresponding to an eigenvalue $|\mu|>2\sqrt{d-1}$.

Call $A$ the adjacency matrix of $G'$, and $A_G$ the adjacency matrix of $G$. Then we have
\[
g^*Ag=g_G^*A_Gg_G+g_{H'}^*Ag_{H'}-2\sum_{i=1}^k g(a_{i,1})g(a_{i,2})+\sum_{\substack{u\in Q\cup R\\a_{i,j}\in M\\u\sim a_{i,j}}} g(u) g(a_{i,j})
\]
where $g_G$ and $g_{H'}$ are the projections of $g$ onto $G$ and $H'$, respectively.

We know that 
\[|g_G^*A_Gg_G|\leq 2\sqrt{d-1}||g_G||^2+\frac{d}{n}\left(\sum_{u\in G} g(u)\right)^2\]
by decomposing $g$ into parts parallel and perpendicular to the all ones vector.

By a combination of \pref{lem:inf-graph-specrad} and \pref{cor:subgraph-less-specrad}, the spectral radius of $H'$ is $2\sqrt{d-1}$, and therefore we have
\[
    |g_G^*A_Gg_G|+|g_{H'}^*A g_{H'}|\leq 2\sqrt{d-1}||g||^2+\frac{d}{n}\left(\sum_{u\in H'} g(u)\right)^2
\]
as $\sum_G g(u)=-\sum_{H'}g(u)$, considering $g\perp \textbf 1$. 

To show that $|\mu|=2\sqrt{d-1}+O(1/\log n)$, we then need to show
\begin{align}
    \frac{1}{\|g\|^2}\left(\frac{d}{n}\left(\sum_{H'} g_{H'}(u)\right)^2-2\sum_{i=1}^k g(a_{i,1})g(a_{i,2})+\sum_{\substack{u\in Q\cup R\\a_{i,j}\in M\\u\sim a_{i,j}}} g(u)g(a_{i,j})\right) = O\left(\frac{1}{\log n}\right). \label{eq:small-interface}
\end{align}

The first term of \pref{eq:small-interface} can be bounded as 
\begin{equation}
\frac dn\left(\sum_{H'} g_{H'}(u)\right)^2\leq \frac{d}{n}\left|H'\right|\|g_{H'}\|^2\leq \frac{\gamma(2+(d-1)(d-2))d}{2n}\|g_{H'}\|^2. \label{eq:first-term-above}
\end{equation}

The second term we can bound as
\begin{equation}
\left|2\sum_{i=1}^k g(a_{i,1})g(a_{i,2})\right|\leq \sum_{a_{i,j}\in M} g(a_{i,j})^2.    \label{eq:second-term-above}
\end{equation}

Now we will bound the last term of \pref{eq:small-interface} using Cauchy Schwarz.
\begin{eqnarray}
\left|\sum_{\substack{u\in Q\cup R\\a_{i,j}\in M\\u\sim a_{i,j}}} g(u)g(a_{i,j})\right|
\leq \sqrt{(d-1)\sum_{u\in Q\cup R} g(u)^2}\sqrt{\sum_{a_{i,j}\in M} g(a_{i,j})^2}. \label{eq:last-term-above}
\end{eqnarray}

We use the following lemma to bound the right hand sides of \pref{eq:second-term-above} and \pref{eq:last-term-above}. The lemma is a generalized version of \cite[Lemma 5.1]{Kah}. The result follows from the same proof, which we reproduce in the appendix for completeness. Here, for two vectors $a,b\in \R^n$, $a\leq b$ if $\forall i\in[n], a(i)\leq b(i)$.

\begin{lemma}[Lemma 5.1 of \cite{Kah}] \label{lem:kahdisp}
Consider a graph on a vertex set $W$, a subset $X$ of $W$, a positive integer $h$, and $s\in L^2(W)$. Let $X_i$ be the set of nodes at distance $i$ from $X$. Assume the following conditions hold:
\begin{enumerate}[(1)]
    \item
    For $h-1\leq i,j\leq h$, all nodes in $X_i$ have the same number of neighbors in $X_j$.
    \item
    If $u\in X_{h-1}$ and $v\in X_{h}$ and $u\sim v$, then  $s(u)/s(v)$ does not depend on the choices of $u$ and $v$. 
    \item
    $s$ is nonnegative and $As\leq \mu s$ on $\textnormal{Ball}_{h-1}(X)$, where $\mu$ is a positive real number.
\end{enumerate}

Then for any $g\in L^2(W)$ such that $|Ag(u)|=\mu|g(u)|$ for $u\in \textnormal{Ball}_{h-1}(X)$, we have
\begin{equation}\label{eq:kahlemma}
\frac{\sum_{v\in X_h} g(v)^2}{\sum_{v\in X_h}s(v)^2}\geq \frac{\sum_{v\in X_{h-1}} g(v)^2}{\sum_{v\in X_{h-1}}s(v)^2}.
\end{equation}
\end{lemma}

To use the lemma, we set $X_0=U\cup V$, and $h$ will vary from $2\leq h\leq \lfloor r/2\rfloor$. Assuming that the girth of $G'$ is at least $r$, the $\lfloor r/2\rfloor$ neighborhoods of each vertex do not overlap. 

Our test vector decays exponentially, with a small adjustment.

\[
s(y)=\left\{
\begin{array}{cc}
     \frac{1}{(d-1)^{h/2}}&X_{h,U}  \\
     \frac{2}{\sqrt{d-1}}-\frac{1}{(d-1)^{3/2}}& y\in  X_{0,V}\\
     \left(\frac{2}{d-2}-\frac{2}{(d-1)(d-2)}\right)\frac{1}{(d-1)^{(h-1)/2}}&y\in X_{h,V}, h\geq 1.
\end{array}\right.
\]
For this assignment of values we have $As\leq (2\sqrt{d-1})s$. In fact, this inequality is sharp at all coordinates except for $y\in X_{1,V}$.

For this $s$, we have that
$
\sum_{y\in X_h} s(y)^2
$
is constant for $h=1,\ldots, \lfloor r/2\rfloor$. Also, recall $Q\cup R=X_1$ and $M=X_2$. By \pref{lem:kahdisp}, as $g$ corresponds to an eigenvalue $|\mu|>2\sqrt{d-1}$, the mass on each of first 2 layers of $X$ can only be at most $2/(r-2)$ of the total mass. 

Combining \pref{eq:first-term-above}, \pref{eq:second-term-above}, and \pref{eq:last-term-above}, we can bound \pref{eq:small-interface} as
\begin{eqnarray*}
\pref{eq:small-interface}
&\leq& \frac{\gamma(2+(d-1)(d-2))d}{2n}\|g_{H'}\|^2+\sum_{a\in X_2} g(a)^2+\sqrt{(d-1)\sum_{u\in X_1} g(u)^2}\sqrt{\sum_{a\in X_2} g(a)^2}\\
&\leq& \left(\frac{\gamma(2+(d-1)(d-2))d}{2n}+(1+\sqrt{d-1})\frac{2}{r-2}\right)\|g\|^2.
\end{eqnarray*}
If we set $\gamma=n^{1/3}$ and $r=\frac{2}3\log_{d-1} n-O_n(1)$, for fixed $d$ this becomes 
\[
O\left(\frac{1}{\log n}\right)\|g\|^2,
\]
meaning that $\mu\leq 2\sqrt{d-1}+O(1/\log n)$. This also gives the desired bounds on vertex expansion and girth, by setting $U=U_m$. Because $|V'|=(1+o_n(1))n$, the bounds on $\Psi(U_m)$, $g(G')$ and $\lambda(G')$ given in terms of $n$ do not change when they are given in terms of $m$.
\end{proof}

\section{Lossless expansion of small sets}
In this section, we prove that sufficiently small sets in a high-girth spectral expander expand losslessly.
\begin{theorem}
    Let $G$ be a $d$-regular graph on $n$ vertices with girth at least $2\alpha\log_{d-1} n + 4$.  Then for any set $S$ with $n^{\kappa}$ vertices,
    \[
        \frac{|\partial S|}{|S|} \ge d - \lambda(G) - \frac{d^{\kappa/\alpha}}{2} - \frac{d}{n^{1-\kappa}}.
    \]
\end{theorem}

\begin{proof}
    Let $S$ be a set of vertices of size $n^{\kappa}$ in $G$.  Let $e_S$ denote the number of internal edges within $S$.  Let $n_i$ denote the number of vertices in $\partial S$ that have $i$ edges from $S$ incident to it.  Then:
    $|\partial S| = n_1 + n_2 + \dots + n_d$ and $|E(S,\partial S)| = n_1 + 2n_2 + \dots + dn_d$.  Note that $|E(S,\partial S)|$ is also equal to $d|S|-2e_S$.  Now consider the graph $H_S$ on vertex set $S$ and edge set given by induced edges on $S$ along with new edges introduced by adding an arbitrary spanning tree for every set of $i$ vertices that are neighbors of a vertex in $\partial S$ with exactly $i$ neighbors in $S$.
    The number of edges in $H(S)$ is equal to
    \[
        e_S + n_2 + 2n_3 + \dots + (d-1)n_d = e_S + |E(S,\partial{S})| - |\partial S| = d|S| - e_S - |\partial S|
    \]
    and $g(H(S))\ge \frac{1}{2}g(G)\ge\alpha\log_{d-1} n + 2$.  As a consequence of the expander mixing lemma (\pref{lem:expander-mixing-lemma}),
    \[
        e_S \le \left(\lambda(G) + \frac{d|S|}{n}\right)|S|
    \]
    for some absolute constant $c$. Consequently,
    \[
        |E(H(S))| \ge \left(d-\lambda(G)-\frac{d|S|}{n}\right)|S| - |\partial S|,
    \]
    which means the average degree is lower bounded by
    \[
        2\left(d-\lambda(G)-\frac{d|S|}{n}-\frac{|\partial S|}{|S|}\right).
    \]    
    Thus by the irregular Moore bound (\pref{lem:irreg-moore-bound}),
    \[
        g(H(S))\le\frac{2\log n^{\kappa}}{\log \left(2\left(d-\lambda(G)-\frac{d|S|}{n}-\frac{|\partial S|}{|S|}\right)-1\right)} + 2
    \]
    and hence
    \[
        \frac{\alpha}{\log (d-1)} \le \frac{2\kappa}{\log \left(2\left(d-\lambda(G)-\frac{d|S|}{n}-\frac{|\partial S|}{|S|}\right)-1\right)}.
    \]
    This implies
    \[
        d-\lambda(G)-\frac{d|S|}{n}-\frac{|\partial S|}{|S|} - \frac{1}{2} \le \frac{d^{2\kappa/\alpha}}{2},
    \]
    and finally by rearranging the above and plugging in $|S| = n^{\kappa}$
    \[
        \frac{|\partial S|}{|S|} \ge d - \lambda(G) - \frac{d^{2\kappa/\alpha}-1}{2}  - \frac{d}{n^{1-\kappa}}.
    \]
\end{proof}

\begin{remark}
    If $G$ is a $n$-vertex $d$-regular Ramanujan graph with girth $\frac{4}{3}\log_{d-1}n$ (which is a condition satisfied by the Ramanujan graphs of \cite{LPS}) then for every set $S$ of size $n^{\kappa}$ for $\kappa<1/3$,
    \[
        \frac{|\partial S|}{|S|} \ge d(1 - o_d(1)).
    \]
\end{remark}

\section*{Acknowledgements}
We would like to thank Shirshendu Ganguly and Nikhil Srivastava for their highly valuable insights, intuition, and comments.

\bibliographystyle{alpha}
\bibliography{ref}

\begin{thebibliography}{CRVW02}

\bibitem[AFH15]{AFH}
Omer Angel, Joel Friedman, and Shlomo Hoory.
\newblock The non-backtracking spectrum of the universal cover of a graph.
\newblock {\em Transactions of the American Mathematical Society},
  367(6):4287--4318, 2015.

\bibitem[AGS19]{AGS}
Noga Alon, Shirshendu Ganguly, and Nikhil Srivastava.
\newblock High-girth near-{R}amanujan graphs with localized eigenvectors.
\newblock {\em arXiv preprint arXiv:1908.03694}, 2019.

\bibitem[AHL02]{AHL}
Noga Alon, Shlomo Hoory, and Nathan Linial.
\newblock The {M}oore bound for irregular graphs.
\newblock {\em Graphs and Combinatorics}, 18(1):53--57, 2002.

\bibitem[ALM15]{AL}
Nalini Anantharaman and Etienne Le~Masson.
\newblock Quantum ergodicity on large regular graphs.
\newblock {\em Duke Math. J.}, 164(4):723--765, 2015.

\bibitem[Alo86]{Che1}
Noga Alon.
\newblock Eigenvalues and expanders.
\newblock {\em Combinatorica}, 6:83--96, 1986.

\bibitem[Alo20]{Alon}
Noga Alon.
\newblock Explicit expanders of every degree and size.
\newblock {\em arXiv preprint arXiv:2003.11673}, 2020.

\bibitem[Ana15]{Anth}
Nalini Anantharaman.
\newblock Quantum ergodicity on regular graphs.
\newblock {\em arXiv preprint arXiv:1512.06624}, 2015.

\bibitem[BLML16]{BLL}
Shimon Brooks, Etienne Le~Masson, and Elon Lindenstrauss.
\newblock Quantum ergodicity and averaging operators on the sphere.
\newblock {\em International Mathematics Research Notices}, 19:6034--6064,
  2016.

\bibitem[Bor19]{Bor}
Charles Bordenave.
\newblock A new proof of {F}riedman's second eigenvalue theorem and its
  extension to random lifts.
\newblock In {\em Annales scientifiques de l'Ecole normale sup{\'e}rieure},
  2019.

\bibitem[CRVW02]{CRVW}
Michael Capalbo, Omer Reingold, Salil Vadhan, and Avi Wigderson.
\newblock Randomness conductors and constant-degree lossless expanders.
\newblock In {\em Proceedings of the 34th Annual ACM Symposium on Theory of
  Computing}, pages 659--668, 2002.

\bibitem[Fri03]{Fri}
Joel Friedman.
\newblock A proof of {A}lon's second eigenvalue conjecture.
\newblock In {\em Proceedings of the thirty-fifth annual ACM symposium on
  Theory of computing}, pages 720--724, 2003.

\bibitem[GLR08]{GLR}
Venkatesan Guruswami, James Lee, and Alexander Razborov.
\newblock Almost euclidean subspaces of $\ell^n_1$ via expander codes.
\newblock In {\em Proceedings of the nineteenth annual ACM-SIAM symposium on
  Discrete algorithms}, pages 353--362, 2008.

\bibitem[Gri01]{Gri}
Dima Grigoriev.
\newblock Linear lower bound on degrees of positivstellensatz calculus proofs
  for the parity.
\newblock {\em Theoretical Computer Science}, 259(1-2):613--622, 2001.

\bibitem[GS18]{GS}
Shirshendu Ganguly and Nikhil Srivastava.
\newblock On non-localization of eigenvectors of high girth graphs.
\newblock {\em arXiv preprint arXiv:1803.08038. To appear in International
  MathematicsResearch Notices}, 2018.

\bibitem[HLW18]{HLW}
Shlomo Hoory, Nathan Linial, and Avi Wigderson.
\newblock Expander graphs and their applications.
\newblock {\em Bull. Amer. Math. Soc.}, 43(4):439–561, 2018.

\bibitem[Kah95]{Kah}
Nabil Kahale.
\newblock Eigenvalues and expansion of regular graphs.
\newblock {\em Journal of the ACM (JACM)}, 42(5):1091--1106, 1995.

\bibitem[LMSS01]{LMSS}
Michael Luby, Michael Mitzenmacher, M.~Amin Shokrollahi, and Daniel Spielman.
\newblock Improved low-density parity-check codes using irregular graphs.
\newblock {\em IEEE Trans. Inform. Theory}, 42(2):585--598, 2001.

\bibitem[LPS88]{LPS}
Alexander Lubotzky, Ralph Phillips, and Peter Sarnak.
\newblock Ramanujan graphs.
\newblock {\em Combinatorica}, 8:261--277, 1988.

\bibitem[LS19]{LS}
Nati Linial and Michael Simkin.
\newblock A randomized construction of high girth regular graphs.
\newblock {\em arXiv preprint arXiv:1911.09640}, 2019.

\bibitem[Mar88]{Mar88}
Grigorii~Aleksandrovich Margulis.
\newblock Explicit group-theoretical constructions of combinatorial schemes and
  their application to the design of expanders and concentrators.
\newblock {\em Problemy peredachi informatsii}, 24(1):51--60, 1988.

\bibitem[MOP20]{MOP}
Sidhanth Mohanty, Ryan O'Donnell, and Pedro Paredes.
\newblock Explicit near-ramanujan graphs of every degree.
\newblock In {\em Proceedings of the 52nd Annual ACM SIGACT Symposium on Theory
  of Computing}, pages 510--523, 2020.

\bibitem[Mor94]{Mor}
Moshe Morgenstern.
\newblock Existence and explicit constructions of q+ 1 regular ramanujan graphs
  for every prime power q.
\newblock {\em Journal of Combinatorial Theory, Series B}, 62(1):44--62, 1994.

\bibitem[Par20]{PP}
Pedro Paredes.
\newblock Spectrum preserving short cycle removal on regular graphs.
\newblock {\em arXiv preprint arXiv:2002.07211}, 2020.

\bibitem[Que96]{Que96}
Gregory Quenell.
\newblock Notes on an example of {McLaughlin}.
\newblock 1996.

\bibitem[Sch08]{Sch}
Grant Schoenebeck.
\newblock Linear level lasserre lower bounds for certain k-csps.
\newblock In {\em 2008 49th Annual IEEE Symposium on Foundations of Computer
  Science}, pages 593--602. IEEE, 2008.

\bibitem[Spi96]{Spiel}
Daniel Spielman.
\newblock Linear-time encodable and decodable error-correcting codes.
\newblock {\em IEEE Trans. Inform. Theory}, 42(6, part 1):1723--1731, 1996.

\bibitem[SS96]{SipSpiel}
Michael Sipser and Daniel Spielman.
\newblock Expander codes.
\newblock {\em IEEE Trans. Inform. Theory}, 42(6, part 1):1710--1722, 1996.

\end{thebibliography}
\appendix
\section{Proof of \pref{lem:kahdisp}}

\begin{proof}
Let $A$ be the adjacency matrix of $W$. Let $P_{h-1}$ and $P_h(X)$  be the orthogonal projections onto $X_{h-1}$ and onto $X_{h}$, respectively. Let $P_{\leq h-1}$ and $P_{\leq h}(X)$ be the orthogonal projections onto $\textnormal{Ball}_{h-1}(X)$ and $\textnormal{Ball}_{h}(X)$, respectively. 
We need to show that

\[
\frac{\|P_{h}g\|^2}{\|P_{h}s\|^2}\geq\frac{\|P_{h-1}g\|^2}{\|P_{h-1}s\|^2}.
\]
Call $A_h=P_{\leq h}AP_{\leq h}$ (so $A_h$ performs the adjacency operator on $\textnormal{Ball}_{h}(X)$). By the conditions of the lemma, we know that there are constants $\alpha,\beta$ and $\gamma$ such that
\begin{equation}\label{eq:kahassume3}
P_hA_hs=\gamma P_hs
\end{equation}
and 
\begin{equation}\label{eq:kahassume2}
    A_hP_hs=\alpha P_hs+\beta P_{h-1}s.
\end{equation}

By assumption, 
\begin{equation}\label{eq:kahassume}
A_h s\leq \mu P_{\leq h-1}s+\gamma P_h s.
\end{equation}
Therefore by applying $P_{\leq h-1}$ to both sides of \pref{eq:kahassume}, 
\begin{eqnarray*}
P_{\leq h-1}A_hs&\leq& \mu P_{\leq h-1}s\\
&\leq&\mu P_{\leq h}s-\mu P_h s.
\end{eqnarray*}
Now we apply $A_h$ to both sides:
\begin{align*}
A_hP_{\leq h-1}A_hs&\leq \mu A_hs-\mu A_hP_{h}s\\
&\leq\mu A_hs-\mu(\alpha P_hs+\beta P_{h-1}s) & \textnormal{by }\pref{eq:kahassume2}\\
&\leq\left(\mu^2 P_{\leq h-1}+\mu(\gamma-\alpha)P_{h}-\mu\beta P_{h-1}\right)s.& \textnormal{by }\pref{eq:kahassume}
\end{align*}

Define the matrix $B:=\mu^2 P_{\leq h-1}+\mu(\gamma-\alpha)P_{h}-\mu\beta P_{h-1}-A_hP_{h-1}A_h$. $B$ has no positive entries on the off diagonal.  Take any eigenvector $\psi$ of $B$. Without loss of generality assume that $\psi$ has a positive entry. Then take $i=\argmax_u\psi(u)/s(u)$. As $\psi\leq (\psi(i)/s(i))s$, $(B\psi)(i)\geq B(\psi(i)/s(i))s)(i)$. The quantity on the right is nonnegative, meaning that the eigenvalue with eigenvector $\psi$ is nonnegative. As $\psi$ was arbitrary, $B$ is positive semidefinite. 

Because $B$ is positive semidefinite,
\begin{equation}\label{eq:quadform}
    g^*A_hP_{\leq h-1}A_hg\leq g^*\left(\mu^2 P_{\leq h-1}+\mu(\gamma-\alpha)P_{h}-\mu\beta P_{h-1}\right)g.
\end{equation}

 For any orthogonal projection $P$, $P^2=P$. Therefore $g^*A_hP_{\leq h-1}A_hg=\|P_{\leq h-1}A_hg\|^2$. Moreover \pref{eq:quadform} becomes
\[
\|P_{\leq h-1}A_hg\|^2\leq \mu^2\|P_{\leq h-1}g\|^2+\mu(\gamma-\alpha)\|P_hg\|^2-\mu\beta\|P_{h-1}g\|^2.
\]
By assumption, $\|P_{\leq h}A_hg\|=\mu\|P_{\leq h}g\|$. Therefore 
\begin{equation}\label{eq:simpquadform}
    (\gamma-\alpha)\|P_hg\|^2\geq \beta\|P_{h-1}g^2\|.
\end{equation}

Moreover, as $A_h$ and $P_h$ are self adjoint, $s^*A_hP_hs=s^*P_hA_hs$, so $\alpha\|P_hs\|^2+\beta\|P_{h-1}\|^2=\gamma\|P_hs\|$. Combining this with \pref{eq:simpquadform}, we obtain \pref{eq:kahlemma}.
\end{proof}

\end{document}